\documentclass[11pt]{article}

\usepackage{amsmath}
\usepackage{amsfonts}
\usepackage{amssymb}
\usepackage[all,2cell]{xy}
\UseTwocells

\newtheorem{definition}{Definition}
\newtheorem{proposition}{Proposition}
\newtheorem{lemma}{Lemma}

\newtheorem{theorem}{Theorem}

\newenvironment{proof}{{\bf Proof:}}{$\text{ }\blacksquare$}

\begin{document}

\title{Derived manifolds and Kuranishi models}
\author{Dennis Borisov\\ Mathematical Institute, University of Oxford\\ dennis.borisov@gmail.com}
\date{\today}
\maketitle

\begin{abstract}
\noindent A model structure is defined on the category of derived $C^\infty$-schemes, and it is used to analyse the truncation $2$-functor from derived manifolds to $d$-manifolds. It is proved that the induced $1$-functor between the homotopy categories is full and essentially surjective, giving a bijection between the sets of equivalence classes of objects. An example is constructed, showing that this $1$-functor is not faithful.
\end{abstract}

\section{Introduction}

Having two maps between $C^\infty$-manifolds
	$$\xymatrix{F,G:X\ar@<-.5ex>[r]\ar@<+.5ex>[r] & Y}$$
one is often interested in the locus where $F=G$. In general this locus is not a manifold, and this leads to the theory of derived manifolds (\cite{Sp10},\cite{BN11}).

There is a special class of such pairs of maps, that is particularly important when working with moduli spaces. In these situations $Y$ is a linear bundle on $X$, and $F$ is assumed to be the zero section. This linear structure becomes important when one defines morphisms as commutative diagrams 
	$$\xymatrix{Y_1\ar[rr]^\psi && Y_2\\ X_1\ar[u]^{G_1}\ar[rr]_\phi && X_2\ar[u]_{G_2}}$$
requiring that $\psi$ is a linear bundle morphism. This leads to the notion of a Kuranishi model\footnote{We consider only those models that have trivial isotropy groups.} (e.g. \cite{FO99},\cite{FOOO09},\cite{FOOO12},\cite{MW12}).

\smallskip

This requirement that morphisms between Kuranishi models involve morphisms of linear bundles has important implications when one tries to glue such models into global objects. As everywhere in derived geometry, gluing happens up to equivalences, not isomorphisms, and because of linearity these equivalences are defined using derivations, resulting in the $2$-category of $d$-manifolds (\cite{Jo12}).

On the other hand, without requiring linearity of morphisms, one obtains a category tensored over the category of simplicial sets (\cite{BN11}), i.e. an $\infty$-category. Of course one can always truncate the $\infty$-categorical structure to a $2$-categorical one, by taking the fundamental groupoids of the mapping spaces, however, we show in this paper that the result is not equivalent to the $2$-category of $d$-manifolds. Even the truncations to $1$-categories (obtained by taking the connected components of the mapping spaces) are not equivalent.

\smallskip

The truncated $1$-functor from derived manifolds to $d$-manifolds is full, essentially surjective and gives a bijection between the equivalence classes of objects. However, it is not faithful, i.e. by linearising we lose a lot of homotopically non-trivial information.

To analyse this truncation we develop a homotopy theory of derived manifolds, which can be of independent interest. More precisely, we construct a right pseudo-model category structure on the category of derived $C^\infty$-schemes of finite type. This model structure allows us to compute all the necessary homotopy pullbacks and mapping spaces.

\medskip

Here is the structure of the paper. In section \ref{CosimplicialSection} we recall the definition of cosimplicial $C^\infty$-schemes and define the right pseudo-model structure on the subcategory of schemes of finite type. 

In section \ref{EnrichmentSection} we show that this right pseudo-model category can be supplemented with a right action of the category of simplicial sets, thus giving us the ability to compute mapping spaces.

In section \ref{DerivedManifoldsSubsection} we recall the definition of derived manifolds, and compute explicitly the simplicial $C^\infty$-rings of the $0$-loci of sections of vector bundles over manifolds of finite type. The result is rather simple: $\{C^\infty(E^{\times_{\mathcal M}^k})\}_{k\geq 0}$.

In section \ref{The2CategorySection} we recall the definition of the $2$-category of $d$-manifolds from \cite{Jo12}. Our presentation is slightly different from the original: instead of pulling the sheaves back, we push them forward, and instead of using morphisms out of the sheaves of differential forms we use derivations. 

In section \ref{TruncationSection} we describe the truncation $2$-functor. Here the simplicial model structure, that we have developed for derived manifolds, comes very handy.

Finally, in section \ref{DManifoldsSection} we prove the main result: the truncation functor induces a full and essentially surjective $1$-functor between the homotopy categories. The sets of equivalence classes of objects are in bijective correspondence. At the end of the section we provide an example, showing that this $1$-functor is far from being faithful.

\smallskip

{\it Acknowledgements:} The author would like to thank D.Joyce for many fruitful discussions. Also the author is grateful for the hospitality and financial support of Max Planck Institute for Mathematics in Bonn, where the paper was partially written.

\tableofcontents

\section{Derived manifolds}

\subsection{Cosimplicial $C^\infty$-schemes}\label{CosimplicialSection}

Recall (e.g. \cite{MR91}) that \underline{a $C^\infty$-ring} consists of a set $A$, together with operations $A^{\times^n}\rightarrow A$ for all $n\geq 0$, parameterized by smooth functions $\mathbb R^n\rightarrow\mathbb  R$. Equivalently, a $C^\infty$-ring is given by a product preserving functor $\mathfrak E\rightarrow Set$, where $\mathfrak E$ is the category, having $\{\mathbb R^n\}_{n\geq 0}$ as objects, and smooth maps as morphisms.

A morphism of $C^\infty$-rings $A\rightarrow B$ is a set theoretic map, that is compatible with the action of smooth functions. \underline{A simplicial $C^\infty$-ring} is just a simplicial diagram in the category of $C^\infty$-rings, and morphisms of simplicial $C^\infty$-rings are natural transformations. We will denote the categories of $C^\infty$-rings and simplicial $C^\infty$-rings by $C^\infty\mathcal R$ and $SC^\infty\mathcal R$ respectively.
\begin{definition} Let $A_\bullet$ be a simplicial $C^\infty$-ring. \underline{The spectrum of $A_\bullet$}, denoted by $\mathbf{Sp}(A_\bullet)$, is the pair $(Sp(A_\bullet),\mathcal O_{\bullet,Sp(A_\bullet)})$, where
	\begin{equation}Sp(A_\bullet):=Hom_{C^\infty\mathcal R}(\pi_0(A_\bullet),\mathbb R),\end{equation} 
together with Zariski topology, and $\mathcal O_{\bullet,Sp(A_\bullet)}$ is the sheaf of simplicial $C^\infty$-rings on $Sp(A_\bullet)$, obtained by the sheafification of
	\begin{equation}\xymatrix{U\ar@{|->}[r] & \{A_n\{(s^n\mathcal U)^{-1}\}\}_{n\geq 0},}\end{equation}
where $\mathcal U:=\{f\in A_0\text{ s.t. }p([f])\neq 0\;\forall p\in U\}$, $s^n:A_0\rightarrow A_n$ is the $n$-fold degeneration, and $A_n\{(s^n\mathcal U)^{-1}\}$ denotes localization in the category of $C^\infty$-rings (\cite{MR91}).\end{definition}
If $A_\bullet$ is a discrete simplicial finitely generated $C^\infty$-ring, i.e. a constant simplicial diagram on $A=C^\infty(\mathbb R^n)/\mathfrak A$, then $\mathbf{Sp}(A_\bullet)$ is just $(X,\mathcal O_X)$, where $X\subseteq\mathbb R^n$ is the set of zeroes of $\mathfrak A$, and $\mathcal O_X$ is the sheaf of germs of smooth functions around $X\subseteq\mathbb R^n$, modulo germs of functions in $\mathfrak A$. 

There is another way of defining the spectrum of a simplicial $C^\infty$-ring $A_\bullet$. Instead of having one topological space $Hom_{C^\infty\mathcal R}(\pi_0(A_\bullet),\mathbb R)$ equipped with a sheaf of simplicial $C^\infty$-rings, one can have a cosimplicial diagram of $C^\infty$-schemes:
	\begin{equation}\{Hom_{C^\infty\mathcal R}(A_n,\mathbb R),\mathcal O_{Sp(A_n)}\}_{n\geq 0}.\end{equation}
However, we would like to consider such cosimplicial diagrams as weakly equivalent, if they have similar structure in the neighbourhood of $Sp(\pi_0(A_\bullet))$ inside $Sp(A_0)$. Thus it is better to work with only one underlying topological space, and we have the following definition.
\begin{definition} \underline{A cosimplicial $C^\infty$-scheme} is a pair $(X,\mathcal O_{\bullet,X})$, where $X$ is a topological space, and $\mathcal O_{\bullet,X}$ is a sheaf of simplicial $C^\infty$-rings, s.t. locally $(X,\mathcal O_{\bullet,X})$ is isomorphic to spectra of simplicial $C^\infty$-rings.\end{definition}
Notice that for any $p\in X$, the stalk $(\mathcal O_{0,X})_p$ is a local $C^\infty$-ring.\footnote{A $C^\infty$-ring $A$ is local, if it has a unique maximal ideal $\mathfrak A\subset A$, and $A/\mathfrak A\cong\mathbb R$.} Alternatively, one can define a cosimplicial $C^\infty$-scheme as a simplicially $C^\infty$-ringed space, that is locally {\it weakly equivalent} to spectra of simplicial $C^\infty$-rings (\cite{BN11},\cite{Sp10}). In this case only $(\pi_0(\mathcal O_{\bullet,X}))_p$ would be local. In the important situations ($\pi_0(\mathcal O_{\bullet,X})$ being locally of finite type) these two definitions are locally weakly equivalent (\cite{BN11}).

\medskip

We will say that a simplicial $C^\infty$-ring $A_\bullet$ is \underline{of finite type} if $\pi_0(A_\bullet)$ is a finitely generated $C^\infty$-ring. Such rings have particularly nice spectra, as the following proposition shows.
\begin{proposition} (\cite{BN11}) Let $A_\bullet$ be a simplicial $C^\infty$-ring of finite type. Then $Sp(A_\bullet)$ is homeomorphic to a locally closed subset of $\mathbb R^n$ for some $n\geq 0$, $\pi_0(\mathcal O_{\bullet,Sp(A_\bullet)})$ is a sheaf of finitely generated $C^\infty$-rings, and $\mathcal O_{0,Sp(A_\bullet)}$ is soft.\end{proposition}
We will work only with simplicial $C^\infty$-rings of finite type. We will also require that the underlying spaces of the cosimplicial $C^\infty$-schemes are second countable and Hausdorff. All these conditions together imply softness of the structure sheaves. This explains the following definition of separability.
\begin{definition}\begin{itemize}\item A cosimplicial $C^\infty$-scheme $(X,\mathcal O_{\bullet, X})$ is \underline{separated} if $X$ is Hausdorff, and $\mathcal O_{0,X}$ is soft.
\item A cosimplicial $C^\infty$-scheme $(X,\mathcal O_{\bullet,X})$ is \underline{of finite type}, if $\pi_0(\mathcal O_{\bullet,X})$ is a sheaf of finitely generated $C^\infty$-rings.
\item A cosimplicial $C^\infty$-scheme is \underline{locally of finite type}, if locally it is isomorphic to a $C^\infty$-scheme of finite type.\end{itemize}\end{definition}
We will denote the category of cosimplicial $C^\infty$-schemes by $\mathbf{C^\infty Sch}$.\footnote{The morphisms are defined in the standard way. Notice, that any morphism between local $C^\infty$-rings is automatically local.} We will say that $(X,\mathcal O_{\bullet,X})\in\mathbf{C^\infty Sch}$ is compact, second countable etc., if the space $X$ is such. We will denote by $\mathbf G\subset\mathbf{C^\infty Sch}$ the full subcategory, consisting of separated, second countable cosimplicial $C^\infty$-schemes, locally of finite type. We will also denote by $\mathbf G_{ft}\subset\mathbf G$ the full subcategory of cosimplicial $C^\infty$-schemes of finite type.

\smallskip

The following proposition shows that, just as in the usual $C^\infty$-geometry, most cosimplicial $C^\infty$-schemes are affine.
\begin{proposition}\label{TheAdjunction} (\cite{BN11}) Let $\mathcal G_{ft}\subset SC^\infty\mathcal R$ be the full subcategory, consisting of simplicial $C^\infty$-rings of finite type. Then $\Gamma\circ\mathbf{Sp}:\mathcal G_{ft}^{op}\rightarrow\mathcal G_{ft}^{op}$ maps weak equivalences to weak equivalences, and in the adjunction
	\begin{equation}\xymatrix{\Gamma:\mathbf G_{ft}\ar@<+.5ex>[r] & \mathcal G_{ft}^{op}:\mathbf{Sp}\ar@<+.5ex>[l]}\end{equation}
the unit $Id_{\mathbf G_{ft}}\rightarrow\mathbf{Sp}\circ\Gamma$ is an isomorphism.\end{proposition}
Proposition \ref{TheAdjunction} realises $\mathbf G_{ft}$ as a full co-reflective subcategory of $\mathcal G^{op}_{ft}$. The latter is a full subcategory of the category $SC^\infty\mathcal R^{op}$ of all simplicial $C^\infty$-rings, and it is well known (\cite{Qu67}) that $SC^\infty\mathcal R$ is a model category. A morphism $f_\bullet:A_\bullet\rightarrow B_\bullet$ is a \underline{weak equivalence} if it induces bijections
	\begin{equation}\xymatrix{\pi_n(A_\bullet)\ar[r]^{\cong} & \pi_n(B_\bullet), & \forall n\geq 0,}\end{equation}
and $f_\bullet$ is a \underline{fibration} if 
	\begin{equation}\xymatrix{A_\bullet\ar[r] & \pi_0(A_\bullet)\underset{\pi_0(B_\bullet)}\times B_\bullet}\end{equation}
is surjective.\footnote{The right hand side consists of the union of connected components of $\mathcal B_\bullet$, that intersect non-trivially the image of $f_\bullet$.} We would like to transfer this homotopy theory to $\mathbf G_{ft}$.

\smallskip

Since the subcategory $\mathcal G_{ft}^{op}\subset SC^\infty\mathcal R^{op}$ is defined by putting conditions on homotopy types of objects, it is clear that if $A_\bullet\in SC^\infty\mathcal R^{op}$ is weakly equivalent to $B_\bullet\in\mathcal G_{ft}^{op}$, then $A_\bullet\in\mathcal G_{ft}^{op}$. However, it is not true that $\mathcal G_{ft}^{op}$ is a model category in its own right, since it does not have all finite colimits. 

It is typical for geometric situations to have a model structure, but not to have all finite colimits. Such categories are called pseudo-model categories (\cite{TV05}). Our definition is slightly more general than the one in \cite{TV05}, since in $C^\infty$-geometry it can happen that $\Gamma(Sp(A),\mathcal O_{Sp(A)})\ncong A$.
\begin{definition} \underline{A right pseudo-model category} is given by a model category $\mathcal M$, a full subcategory $\mathcal H\subseteq\mathcal M$, closed with respect to finite limits and weak equivalences in $\mathcal M$, and a full co-reflective subcategory $\mathcal P\subseteq\mathcal H$, with the right adjoint to the inclusion denoted by $\mathcal R:\mathcal H\rightarrow\mathcal P$, s.t. $\mathcal R$ maps cofibrations and weak equivalences to cofibrations and weak equivalences respectively.\footnote{A morphism in $\mathcal P$ is a weak equivalence/cofibration, if it is a weak equivalence/cofibration as a morphism in $\mathcal H$.}\end{definition}
If we denote by $\widetilde{\mathcal P}\subseteq\mathcal H$ the full subcategory, consisting of objects that are weakly equivalent to objects in $\mathcal P$, we obviously obtain an adjunction 
	\begin{equation}\xymatrix{\subseteq:\mathcal P\ar@<+.5ex>[r] & \widetilde{\mathcal P}:\mathcal R\ar@<+.5ex>[l]}\end{equation}
whose unit is invertible, and whose counit consists of weak equivalences. 
Thus we see that the simplicial localization of $\mathcal P$ sits inside the simplicial localization of $\mathcal M$ as a full simplicial subcategory. Moreover, as the following proposition shows, $\mathcal P$ inherits a model structure from $\mathcal M$.\footnote{As in \cite{Ho99}, we distinguish between a model structure on a category, and a model category, the latter being a finitely complete and cocomplete category, together with a model structure.} The proof is straightforward.
\begin{proposition}\label{ModelStructure} Let $\mathcal C,\mathcal F,\mathcal W\subseteq\mathcal M$ be the subcategories of cofibrations, fibrations, and weak equivalences respectively. Then $(\mathcal C\cap\mathcal P,\mathcal W\cap\mathcal P)$ define a model structure on $\mathcal P$, with the fibrations being retracts of $\mathcal R(\mathcal F)$.\end{proposition}
By definition $\mathcal G_{ft}^{op}\subset SC^\infty\mathcal R^{op}$ is closed with respect to weak equivalences, and it is easy to check that it is closed with respect to all finite limits. The functor $\mathbf{Sp}:\mathcal G_{ft}^{op}\rightarrow\mathbf G_{ft}$ has a left adjoint $\Gamma$, and from Proposition \ref{TheAdjunction} we know that $\mathbf{Sp}$ maps weak equivalences to weak equivalences, and the unit $Id_{\mathbf G_{ft}}\rightarrow\mathbf{Sp}\circ\Gamma$ is an isomorphism.

Finally, it is straightforward to prove that given a fibration $A_\bullet\rightarrow B_\bullet$ of simplicial $C^\infty$-rings, if we invert every $a\in A_0$, $b\in B_0$, that do not vanish on $Sp(\pi_0(A_\bullet))$, $Sp(\pi_0(B_\bullet))$ respectively, we get a fibration again. This means that $\mathbf{Sp}$ maps cofibrations to cofibrations. Altogether we have the following statement.
\begin{proposition}\label{PseudoModel} The quadruple $(SC^\infty\mathcal R^{op}, \mathcal G^{op}_{ft},\mathbf G_{ft},\mathbf{Sp})$ is a right pseudo-model category.\end{proposition}
The advantage of working with the model structure on $\mathbf G_{ft}$, instead of $SC^\infty\mathcal R$, is that the notion of fibration in $\mathbf G_{ft}$ is better adapted to geometry. Indeed, as the following proposition shows, starting with a fibrant scheme and restricting to a locally closed subset, we get a fibrant scheme again. The proof is straightforward.
\begin{proposition}\label{Induction}\begin{itemize}\item[1.] Let $(X,\mathcal O_{\bullet,X})\in\mathbf G_{ft}$, and let $\iota:Y\subseteq X$ be a locally closed subset. Then $(Y,\iota^{-1}(\mathcal O_{\bullet,X}))\in\mathbf G_{ft}$.
\item[2.] If $(X,\mathcal O_{\bullet,X})$ is fibrant, so is $(Y,\iota^{-1}(\mathcal O_{\bullet,X}))$.\end{itemize}\end{proposition}
An immediate consequence of Proposition \ref{Induction}, is that every manifold $\mathcal M$ of finite type\footnote{Meaning a manifold $\mathcal M$, s.t. $C^\infty(\mathcal M)$ is a finitely generated $C^\infty$-ring. This is equivalent to being able to embed $\mathcal M$ into $\mathbb R^n$ for some finite $n$.}  is fibrant, as an object of $\mathbf G_{ft}$. Indeed, such $\mathcal M$ is embeddable into some $\mathbb R^n$, and then $\mathcal M$ is a retract of its tubular neighbourhood, which is fibrant, since it is an open subscheme of $\mathbb R^n$.

Another example of a fibration is the trivial bundle $\mathbb R^{n+m}\rightarrow\mathbb R^n$, since it is the spectrum of the free $C^\infty$-morphism $C^\infty(\mathbb R^n)\rightarrow C^\infty(\mathbb R^{n+m})$. Consequently, a trivial bundle over any manifold $\mathcal M$ of finite type is a fibration, since it is a retract of a trivial bundle over some $\mathbb R^n$. Then any vector bundle over $\mathcal M$ is a fibration, since any vector bundle is a retract of a trivial one.

\smallskip

We finish this subsection by noting that any inclusion of a closed cosimplicial $C^\infty$-subscheme is a cofibration, since it corresponds to a surjective morphism of simplicial $C^\infty$-rings. In particular, every $\mathbf X\in\mathbf G_{ft}$ is cofibrant. Note however, that there are cofibrations in $\mathbf G_{ft}$, that are not inclusions of closed subschemes. For example any smooth morphism between manifolds is a cofibration.

\subsection{Enrichment in $SSet$}\label{EnrichmentSection}

The category $SC^\infty\mathcal R$ of simplicial $C^\infty$-rings is a simplicial model category. The \underline{simplicial structure} is defined as follows (\cite{Qu67}, section II.1): for any $A_\bullet\in SC^\infty\mathcal R$, and any $K\in SSet$ we have
	\begin{equation}(A_\bullet\otimes K)_n:=\underset{k_n\in K_n}\coprod A_n,\end{equation}
where the coproduct is taken in the category $C^\infty\mathcal R$ of $C^\infty$-rings. For any weakly order preserving map $f:m\rightarrow n$ in $\Delta$ we have
	\begin{equation}f^*:(A_\bullet\otimes K)_n\rightarrow(A_\bullet\otimes K)_m,\qquad f^*(a_{k_n}):=(f^*(a))_{f^*(k_n)},\end{equation}
where $a\in A_n$, and $a_{k_n}$ belongs to the copy of $A_n$, that is indexed by $k_n\in K_n$.

This gives an enrichment of $SC^\infty\mathcal R$ in $SSet$ by adjunction:
	\begin{equation}\underline{Hom}(A_\bullet,B_\bullet):=\{Hom(A_\bullet\otimes\Delta[k],B_\bullet)\}_{k\geq 0},\end{equation}
and there is the second left adjoint to $\underline{Hom}(-,-)$:
	\begin{equation}\underline{Hom}(A_\bullet,B_\bullet)\cong\{Hom(A_\bullet,B_\bullet^{\Delta[k]})\}_{k\geq 0}.\end{equation}

\smallskip

The category $\mathbf G_{ft}$ does not inherit a simplicial structure from $SC^\infty\mathcal R^{op}$. However, it does inherit a part of it, which is enough for doing homotopical computations. First we give such categories a name. We will say that a simplicial set $S_\bullet$ is \underline{of finite type}, if $\pi_0(S_\bullet)$ is finite. Let $SSet_{ft}\subset SSet$ be the full subcategory, consisting of simplicial sets of finite type.
\begin{definition} \underline{A right simplicial category} is a category $\mathcal P$, together with bifunctors
	\begin{equation}\xymatrix{\underline{Hom}(-,-):\mathcal P^{op}\times\mathcal P\ar[r] & SSet,}\end{equation}
	\begin{equation}\xymatrix{-^-:\mathcal P^{op}\times SSet_{ft}\ar[r] & \mathcal P,}\end{equation}
s.t. $\forall K,L\in SSet_{ft}$, $\forall A,B\in\mathcal P$ there are natural coherent isomorphisms
	\begin{equation}A^{pt}\cong A,\; (A^K)^L\cong A^{K\times L},\; 
	Hom(A,B^K)\cong Hom(K,\underline{Hom}(A,B)).\end{equation}
\end{definition}
Just as with the usual simplicial categories, $\underline{Hom}(-,-)$ makes $\mathcal P$ into a $SSet$-category, and $Hom_{\mathcal P}(-,-)\cong\underline{Hom}(-,-)_0$. Similar to the full simplicial case, there is a notion of compatibility between the structures of a right simplicial category and a right pseudo-model category, as follows.
\begin{definition} \underline{A right simplicial pseudo-model category} consists of\begin{itemize}
\item a simplicial model category $\mathcal M$,
\item a right pseudo-model category $(\mathcal M,\mathcal H,\mathcal P,\mathcal R)$, extending the model structure on $\mathcal M$, and such that $\forall K\in SSet_{ft}$, $\forall A\in\mathcal H$
	\begin{equation}A^K\in\mathcal H,\end{equation}
\item a right simplicial structure on $\mathcal P$, s.t. $\forall A\in\mathcal H$, $\forall K\in SSet_{ft}$ there is a coherent natural isomorphism
	\begin{equation}\label{PreservingSimplicial}\mathcal R(A^K)\cong(\mathcal R(A))^K.\end{equation}
\end{itemize}\end{definition}
Just as in the full simplicial case, a right simplicial structure on a right pseudo-model category lets us calculate mapping spaces. For $A,B\in\mathcal P$ we define
	\begin{equation}\underline{Hom}(A,B):=\{Hom_{\mathcal P}(A,B^{\Delta[n]})\}_{n\geq 0}.\end{equation}
The following proposition shows that these mapping spaces behave as expected. The proof is standard.
\begin{proposition}\begin{itemize}\item[1.] Let $(\mathcal M,\mathcal H,\mathcal P,\mathcal R)$ be a right simplicial pseudo-model category. Let $j:A\rightarrow B$, $q:X\rightarrow Y$ be a cofibration and a fibration in $\mathcal P$, respectively. Then
	\begin{equation}\xymatrix{j*q:\underline{Hom}(B,X)\ar[r] & \underline{Hom}(A,X)\underset{\underline{Hom}(A,Y)}\times
	\underline{Hom}(B,Y)}\end{equation}
is a fibration of simplicial sets, which is trivial if $j$ or $q$ is trivial.
\item[2.] Let $A,X\in\mathcal P$, with $A$ being cofibrant and $X$ being fibrant. Then $\underline{Hom}(A,X)$ is a Kan complex, and it is weakly equivalent to the corresponding mapping space in the simplicial localization of $\mathcal P$.\end{itemize}\end{proposition}
We would like to define a right simplicial structure on $\mathbf G_{ft}$, s.t. it is compatible with the right pseudo-model structure from Proposition \ref{PseudoModel}. In fact there is a natural right simplicial structure on all of $\mathbf{C^\infty Sch}$, defined as follows: $\forall K\in SSet_{ft}$, $\forall(X,\mathcal O_{\bullet,X})\in\mathbf{C^\infty Sch}$ let
	\begin{equation}\label{Mapping}(X,\mathcal O_{\bullet, X})^K:=
	\underset{i\in\pi_0(K)}\prod(X,\mathcal O_{\bullet,X}\otimes K_i),\end{equation}
where $K_i\subseteq K$ is the connected component corresponding to $i\in\pi_0(K)$, $\mathcal O_{\bullet,X}\otimes{K_i}$ is the sheaf of simplicial $C^\infty$-rings, generated by 
	\begin{equation}\xymatrix{U\ar@{|->}[r] & \Gamma(U,O_{\bullet,X})\otimes{K_i},}\end{equation} 
and the product is taken in $\mathbf{C^\infty Sch}$. We claim that $(X,\mathcal O_{\bullet,X})^K\in\mathbf{C^\infty Sch}$. This is an immediate consequence of the following lemma.
\begin{lemma} For any $A_\bullet\in SC^\infty\mathcal R$, any $K\in SSet_{ft}$, s.t. $\pi_0(K)=pt$, there is a natural isomorphism
	\begin{equation}\label{RightSimplicial}\xymatrix{\mathbf{Sp}(A_\bullet\otimes K)\ar[r] & \mathbf{Sp}(A_\bullet)^K.}\end{equation}\end{lemma}
\begin{proof} Consider $A_\bullet\otimes K\rightarrow A_\bullet$, given by the unique $K\rightarrow\Delta[0]$. It is easy to see that this map induces an isomorphism $\pi_0(A_\bullet\otimes K)\rightarrow\pi_0(A_\bullet)$. Therefore, we have a chosen homeomorphism
	\begin{equation}\xymatrix{Sp(A_\bullet)\ar[r]^{\cong\quad} & Sp(A_\bullet\otimes K).}\end{equation}
Clearly, it is natural in $A_\bullet$ and $K$, and we will use it to identify $Sp(A_\bullet)$ and $Sp(A_\bullet\otimes K)$. We will denote both of them by $Sp(A_\bullet)$.

Let $O$ be the following pre-sheaf of simplicial $C^\infty$-rings on $Sp(A_\bullet)$:
	\begin{equation}\xymatrix{U\ar@{|->}[r] & \{A_n\{(s^n\mathcal U)^{-1}\}\}_{n\geq 0},}\end{equation}
where $\mathcal U\subset A_0$ consists of elements, that do not vanish on $U$. Also let $O'$ be the pre-sheaf
	\begin{equation}\xymatrix{U\ar@{|->}[r] & \{(A_\bullet\otimes K)_n\{(s^n\mathcal U')^{-1}\}\}_{n\geq 0},}\end{equation}
where $\mathcal U'\subset (A_\bullet\otimes K)_0$ consists of the elements, that do not vanish on $U$. We claim that there is a natural morphism of pre-sheaves
	\begin{equation}\label{PreMorphism}\xymatrix{O\otimes K\ar[r] & O'.}\end{equation}
Indeed, let $k\in K_n$, let $\iota_k:A_n\rightarrow\underset{K_n}\coprod A_n$ be the corresponding inclusion, and let $\widetilde{\iota}_k:A_n\rightarrow(\underset{K_n}\coprod A_n)\{(s^n\mathcal U')^{-1}\}$ be the composition with localization. It is obvious that $\widetilde{\iota}_k$ inverts every element of $s^n(\mathcal U)$, and hence we have 
	\begin{equation}\xymatrix{\underset{K_n}\coprod(A_n\{(s^n\mathcal U)^{-1}\})\ar[r] &
	(\underset{K_n}\coprod A_n)\{(s^n\mathcal U)^{-1}\}},\end{equation}
giving us (\ref{PreMorphism}). Sheafification of (\ref{PreMorphism}) produces a morphism of sheaves
	\begin{equation}\xymatrix{\mathcal O_{\bullet,Sp(A_\bullet)}\otimes K\ar[r] & 
	\mathcal O_{\bullet,A_\bullet\otimes K}.}\end{equation}
It is straightforward to see that this morphism is an isomorphism on stalks.\end{proof}

\smallskip

Since $\mathbf{Sp}$ is a right adjoint, it is clear that (\ref{RightSimplicial}) is an isomorphism for any $K\in SSet_{ft}$. It is also easy to see that $(X,\mathcal O_{\bullet,X})^K\in\mathbf G_{ft}$, for any $(X,\mathcal O_{\bullet,X})\in\mathbf G_{ft}$, and $K\in SSet_{ft}$. Thus we have proved the following proposition.
\begin{proposition}\label{SimplicialStructure} The right pseudo-model category $(SC^\infty\mathcal R^{op},\mathcal G_{ft}^{op},\mathbf G_{ft},\mathbf{Sp})$ is right simplicial, with the simplicial structure given by (\ref{Mapping}).\end{proposition}

\subsection{Derived manifolds}\label{DerivedManifoldsSubsection}

Using the weak equivalences in $\mathbf G_{ft}$, we have the notion of homotopy limits,\footnote{ They are defined as limits in the simplicial localization of $\mathbf G_{ft}$.} and since $\mathbf G_{ft}$ is a right simplicial pseudo-model category, these homotopy limits are computable. Here we compute a particular kind of homotopy limit, that is essential for the notion of derived manifolds.
\begin{definition}\begin{itemize}\item \underline{A Kuranishi neighbourhood} is a homotopy limit in $\mathbf G_{ft}$ of a diagram
	\begin{equation}\label{KuranishiSpace}\xymatrix{\mathbb R^n\ar@<+.5ex>[r]^\omega\ar@<-.5ex>[r]_\sigma & 
	\mathbb R^{m+n}}\end{equation}
where $\omega$ is the $0$-section, and $\sigma$ is any smooth section.
\item \underline{A derived manifold of finite type} is any $\mathbf X\in\mathbf G_{ft}$, s.t. locally it is weakly equivalent to Kuranishi neighbourhoods.
\item We will denote the full simplicial subcategory of $\mathbf G_{ft}$, consisting of derived manifolds, by $\mathbf{Man}_{ft}$.\end{itemize}\end{definition}
We would like to compute the homotopy limit of (\ref{KuranishiSpace}) explicitly. The right simplicial pseudo-model structure on $\mathbf G_{ft}$ allows us to do this. In fact, the same approach works in a more general situation.

\smallskip

Let $\mathcal M$ be a manifold of finite type,\footnote{Recall that $\mathcal M$ is of finite type, if $C^\infty(\mathcal M)$ is a finitely generated $C^\infty$-ring.} and let $\xi:E\rightarrow\mathcal M$ be a vector bundle over $\mathcal M$. Let $\sigma:\mathcal M\rightarrow E$ be any smooth section, and let $\omega:\mathcal M\rightarrow E$ be the $0$-section. We would like to compute the homotopy equalizer in the following diagram:
	\begin{equation}\label{BigKuranishi}\xymatrix{\mathcal M\ar@<+.5ex>[r]^\omega\ar@<-.5ex>[r]_\sigma & E.}\end{equation}
Since $\omega,\sigma$ are sections of $\xi$, computing the usual equalizer of (\ref{BigKuranishi}) is equivalent to computing the pullback of 
	\begin{equation}\label{KuranishiBig}\xymatrix{& \mathcal M\ar[d]^\omega\\ \mathcal M\ar[r]_\sigma & E}\end{equation}	
in the category $\mathbf G_{ft}/\mathcal M$. Comparing fibrant replacements of (\ref{BigKuranishi}) and (\ref{KuranishiBig}), we see that also homotopy limits of (\ref{BigKuranishi}), (\ref{KuranishiBig}) agree.

We know (Proposition \ref{Induction}, and the discussion immediately after) that $\xi:E\rightarrow\mathcal M$ is a fibration, and hence $E$ is fibrant in $\mathbf G_{ft}/\mathcal M$. Therefore, to compute the homotopy pullback of (\ref{KuranishiBig}) it is enough to find a fibrant resolution of only one of $\omega$, $\sigma$ (e.g. \cite{Lu09}, proposition A.2.4.4). Now we are going to produce a resolution of $\sigma$.

Let $i_0,i_1:\Delta[0]\rightrightarrows\Delta[1]$ be the canonical inclusions. They are trivial cofibrations, and hence, since $E$ is fibrant, the morphism $i_1^*:E^{\Delta[1]}\rightarrow E$ is a trivial fibration. Consider the following pullback diagram in $\mathbf G_{ft}$:
	\begin{equation}\xymatrix{\mathcal M'\ar[d]_j\ar[rr]^{\tau} && E^{\Delta[1]}\ar[d]^{i_1^*}\\ 
	\mathcal M\ar[rr]_\sigma && E}\end{equation}
Obviously $j$ is a trivial fibration, and it has a section (\cite{GJ99}, proof of Lemma II.8.4). Indeed, consider the following diagram
	\begin{equation}\xymatrix{\mathcal M\ar[d]_=\ar[r]^\sigma & E\ar[r]^{\delta^*\quad} & E^{\Delta[1]}\ar[d]^{i_1^*}\\
	\mathcal M\ar[rr]_\sigma && E}\end{equation}
where $\delta:\Delta[1]\rightarrow\Delta[0]$. This diagram commutes because $i_1$ is a section of $\delta$. Therefore there is a unique $l:\mathcal M\rightarrow\mathcal M'$, s.t.
	\begin{equation}j\circ l=Id_\mathcal M,\quad\tau\circ l=\delta^*\circ\sigma.\end{equation}
Moreover, since $i_0$ is also a section of $\delta$, we have $i_0^*\circ\tau\circ l=\sigma$, i.e. the following diagram commutes
	\begin{equation}\xymatrix{\mathcal M\ar[rr]^\sigma\ar[d]_l && E\ar[d]^=\\ 
	\mathcal M'\ar[rr]_{i_0^*\circ\tau} && E}\end{equation} 
One can show (loc. cit.) that $i_0^*\circ\tau$ is a fibration, and therefore $i_0^*\circ\tau$ is a fibrant replacement of $\sigma$ ($l$ is a weak equivalence, since it is a section of a trivial fibration). So to compute the homotopy limit of (\ref{KuranishiBig}) it is enough to compute the usual limit in $\mathbf G_{ft}/\mathcal M$ of
	\begin{equation}\label{GoodKuranishi}\xymatrix{&& \mathcal M\ar[d]^\omega\\ 
	\mathcal M'\ar[rr]_{i_0^*\circ\tau} && E}\end{equation}
Now we would like to write (\ref{GoodKuranishi}) in terms of the $C^\infty$-rings of functions. Let $\omega^*,\sigma^*:C^\infty(E)\rightarrow C^\infty(\mathcal M)$ be the $C^\infty$-morphisms, corresponding to $\omega,\sigma$. 

According to Proposition \ref{SimplicialStructure} 	\begin{equation}E^{\Delta[1]}\cong\mathbf{Sp}(C^\infty(E)\otimes\Delta[1]).\end{equation}
For every $k\geq 0$ the set $\Delta[1]_k$ consists of weakly order preserving maps $\eta:\underline{k}\rightarrow\underline{1}$. There are $k+2$ such maps, and we denote them by $\{\eta_k^0,\ldots,\eta_k^{k+1}\}$, where $\eta_k^i$ sends $\{0,\ldots,i-1\}\mapsto 0$ and $\{i,\ldots,k\}\mapsto 1$. Therefore we have
	\begin{equation}(C^\infty(E)\otimes\Delta[1])_k=
	C^\infty(E)^0\underset\infty\otimes\ldots\underset\infty\otimes C^\infty(E)^{k+1},\end{equation}
where $C^\infty(E)^i$ is a copy of $C^\infty(E)$, corresponding to $\eta_k^i$, $0\leq i\leq k+1$. For every $k\geq 0$ the set $\Delta[0]_k$ consists of the unique map $\zeta_k:\underline{k}\rightarrow\underline{0}$, and by definition
	\begin{equation}i_0(\zeta_k):=\eta_k^0,\quad i_1(\zeta_k):=\eta_k^{k+1}.\end{equation}
Since $\mathbf{Sp}:\mathcal G_{ft}^{op}\rightarrow\mathbf G_{ft}$ has a left adjoint, it preserves limits, and we have
	\begin{equation}\mathcal M'\cong\mathbf{Sp}(\{C^\infty(E)^0\underset\infty\otimes\ldots
	\underset\infty\otimes C^\infty(E)^k\underset\infty\otimes C^\infty(\mathcal M)^{k+1}\}_{k\geq 0}).\end{equation}
Then the limit of (\ref{GoodKuranishi}) is
	\begin{equation}\mathbf{Sp}(\{C^\infty(\mathcal M)^0\underset\infty\otimes C^\infty(E)^1\underset\infty\otimes\ldots
	\underset\infty\otimes C^\infty(E)^k\underset\infty\otimes C^\infty(\mathcal M)^{k+1}\}_{k\geq 0}),\end{equation}
where the simplicial structure maps are given by the structure maps on $\Delta[1]$, identities on $C^\infty(E)$, and $\omega^*,\sigma^*$; with $\omega^*$ being used when $C^\infty(\mathcal M)$ is indexed by $\eta_k^0$, and $\sigma^*$ being used when $C^\infty(\mathcal M)$ is indexed by $\eta_k^{k+1}$. 

We will denote $\{C^\infty(\mathcal M)^0\underset\infty\otimes C^\infty(E)^1\underset\infty\otimes\ldots\underset\infty\otimes C^\infty(E)^k\underset\infty\otimes C^\infty(\mathcal M)^{k+1}\}_{k\geq 0}$ by $\mathcal B(\omega^*,\sigma^*)_\bullet$, since it is clear that this is just the bar construction of $C^\infty(E)$ with coefficients in $C^\infty(\mathcal M)$.

\smallskip

The simplicial $C^\infty$-ring $\mathcal B(\omega^*,\sigma^*)_\bullet$ is not very complicated, but it can be made even simpler, making homotopical calculations so much easier. Consider first the case when $E=\mathcal M\times\mathbb R^m$ is a trivial bundle. Then we have $C^\infty(E)=C^\infty(\mathcal M)\underset\infty\otimes C^\infty(\mathbb R^m)$, and $\omega^*,\sigma^*$ can be equivalently described using
	\begin{equation}\xymatrix{C^\infty(\mathbb R^m)\ar@<+.5ex>[r]^{\mu^*}\ar@<-.5ex>[r]_{\nu_*} & 
	C^\infty(\mathcal M),}\end{equation}
where $\mu^*$ factors through $\mathbb R$, and $\nu^*$ is some $C^\infty$-morphism. Let $\mathcal B(\mu^*,\nu^*)_\bullet$ be the bar construction of $C^\infty(\mathbb R^m)$ with coefficients in $\mathbb R$, $C^\infty(\mathcal M)$, i.e.
	\begin{equation}\mathcal B(\mu^*,\nu^*)_k=C^\infty(\mathbb R^m)^1\underset\infty\otimes\ldots
	\underset\infty\otimes C^\infty(\mathbb R^m)^k\underset\infty\otimes C^\infty(\mathcal M)^{k+1}.\end{equation}
Now it is easy to see that $\mathcal B(\omega^*,\sigma^*)_\bullet$ is a colimit of the following diagram
	\begin{equation}\xymatrix{C^\infty(\mathcal M)\ar[rr]^{(i_1)_*}\ar[d] && C^\infty(\mathcal M)\otimes\Delta[1]\\
	\mathcal B(\mu^*,\nu^*)_\bullet &&}\end{equation}
Applying $\mathbf{Sp}$ we obtain $\mathbf{Sp}(\mathcal B(\omega^*,\sigma^*)_\bullet)$ as a limit of the following diagram
	\begin{equation}\xymatrix{&& \mathcal M^{\Delta[1]}\ar[d]^{i_1^*}\\ \mathbf{Sp}(\mathcal B(\mu^*,\nu^*)_\bullet\ar[rr] &&
	\mathcal M}\end{equation}
Therefore the morphism $\mathbf{Sp}(\mathcal B(\omega^*,\sigma^*)_\bullet)\rightarrow\mathbf{Sp}(\mathcal B(\mu^*,\nu^*)_\bullet)$ is a weak equivalence. Now we notice that 
	\begin{equation}\mathcal B(\mu^*,\nu^*)_\bullet\cong\{C^\infty(E^{\times^k_\mathcal M})\}_{k\geq 0},\end{equation}
and since every bundle is locally trivial, be obtain the following result. 
\begin{proposition}\label{ZeroLocus} Let $E\rightarrow\mathcal M$ be a vector bundle over a manifold $\mathcal M$ of finite type, and let $\sigma:\mathcal M\rightarrow E$ be a section. Let $\omega:\mathcal M\rightarrow E$ be the $0$-section. The homotopy equalizer of
	\begin{equation}\xymatrix{\mathcal M\ar@<+.5ex>[r]^\omega\ar@<-.5ex>[r]_\sigma & E}\end{equation}
can be written as
	\begin{equation}\mathbf{Sp}(\{C^\infty(E^{\times^k_\mathcal M})\}_{k\geq 0}).\end{equation} 
\end{proposition}
By construction, for each $k\geq 2$, the $C^\infty$-ring $C^\infty(E^{\times^k_\mathcal M})$ is obtained as a coproduct of degenerations of $C^\infty(E)$. In other words $\{C^\infty(E^{\times^k_\mathcal M})\}_{k\geq 0}$ is a $1$-skeletal simplicial $C^\infty$-ring. Using this fact it is easy to show that $\mathbf{Sp}(\{C^\infty(E^{\times^k_\mathcal M})\}_{k\geq 0})$ is fibrant. Indeed
	\begin{equation}\xymatrix{C^\infty(\mathcal M)\ar[r] & C^\infty(E)\ar@<+.7ex>[l]\ar@<-.7ex>[l]}\end{equation}
is a retract of some
	\begin{equation}\xymatrix{C^\infty(\mathcal M)\ar[r] & C^\infty(F)\ar@<+.7ex>[l]\ar@<-.7ex>[l],}\end{equation}
where $F\rightarrow\mathcal M$ is a trivial bundle. Since $C^\infty(\mathcal M)\rightarrow C^\infty(F)$ is a free morphism, we see that $\mathbf{Sp}(\{C^\infty(F^{\times^k_\mathcal M})\}_{k\geq 0})$ is fibrant, and hence so is $\mathbf{Sp}(\{C^\infty(E^{\times^k_\mathcal M})\}_{k\geq 0})$.

Going back to Kuranishi neighbourhoods, we would like to fix a particular model for such objects, i.e. a particular choice of the homotopy pullback. The preceding discussion suggests the following definition.
\begin{definition}\label{StandardKuranishi} A \underline{standard Kuranishi neighbourhood} is $\mathbf{Sp}(\{C^\infty(\mathbb R^{n+km})\}_{k\geq 0})$, with the simplicial structure defined as above.\end{definition}

\section{D-manifolds}

\subsection{The $2$-category of $d$-spaces}\label{The2CategorySection}

We define a \underline{$d$-space} to be a quadruple $(X,\mathcal O'_X,\mathcal E_X,d)$, where $X$ is a second countable, Hausdorff space; $\mathcal O'_X$ is a sheaf of $C^\infty$-rings on $X$; $\mathcal E_X$ is a sheaf of $\mathcal O'_X$-modules; and $d:\mathcal E_X\rightarrow\mathcal O'_X$ is a morphism of $\mathcal O'_X$-modules, s.t.\begin{itemize}
\item[1.] $\mathcal O_X:=\mathcal O'_X/d\mathcal E_X$ is a soft sheaf, and its stalks are local $C^\infty$-rings,
\item[2.] $\mathcal O_X$ is locally finitely generated (as a sheaf of $C^\infty$-rings),
\item[3.] $d\mathcal E_X\subset\mathcal O'_X$ is a sheaf of square-zero ideals, and 
	\begin{equation}\label{Gce}(d\mathcal E_X)\mathcal E_X=0.\end{equation}\end{itemize}
It is easy to see that this definition of a $d$-space is equivalent to Definition 4.1.4 in \cite{Jo12}, in the sense that they define the same objects. 
We will denote a $d$-space $(X,\mathcal O'_X,\mathcal E_X,d)$ by $\underline{X}$.

\smallskip

\underline{A $1$-morphism of $d$-spaces} $\underline{X}\rightarrow\underline{Y}$ is given by a triple $\underline{\phi}:=(\phi,\phi',\phi'')$, where $\phi:X\rightarrow Y$ is a continuous map, $\phi':\mathcal O'_Y\rightarrow\phi_*(\mathcal O'_X)$ is a morphism of sheaves of $C^\infty$-rings, and $\phi'':\mathcal E_Y\rightarrow\phi_*(\mathcal E_X)$ is a morphism of $\mathcal O'_Y$-modules, s.t. the following diagram is commutative
	\begin{equation}\label{OneMorphism}\xymatrix{\mathcal E_Y\ar[d]_d\ar[rr]^{\phi''} && \phi_*(\mathcal E_X)\ar[d]^{\phi_*(d)}\\ 
	\mathcal O'_Y\ar[rr]_{\phi'}  && \phi_*(\mathcal O'_X)}\end{equation}
Given two $1$-morphisms $\underline{\phi}_1:\underline{X}\rightarrow\underline{Y}$, $\underline{\phi}_2:\underline{Y}\rightarrow\underline{Z}$, their composition is
	\begin{equation}\underline{\phi}_2\circ\underline{\phi}_1:=
	(\phi_2\circ\phi_1,(\phi_2)_*(\phi_1')\circ\phi_2',(\phi_2)_*(\phi_1'')\circ\phi_2'').\end{equation}
Clearly, commutativity of (\ref{OneMorphism}) implies that $\phi'$ induces a morphism
	\begin{equation}\xymatrix{\phi^\sharp:\mathcal O_Y=\mathcal O'_Y/d\mathcal E_Y\ar[r] &
	\phi_*(\mathcal O'_X)/\phi_*(d)(\phi_*(\mathcal E_X))\cong\phi_*(\mathcal O_X).}\end{equation}	
We will say that $\underline{\phi},\underline{\psi}:\underline{X}\rightarrow\underline{Y}$ are \underline{scheme-theoretically equal} if $\phi=\psi$ and $\phi^\sharp=\psi^\sharp$. In this case the image of $\psi'-\phi'$ lies in $\phi_*(d\mathcal E_X)$, which is a square-zero ideal, hence we have a $C^\infty$-derivation
	\begin{equation}\xymatrix{\psi'-\phi':\mathcal O'_Y\ar[r] & \phi_*(d\mathcal E_X),}\end{equation} 
and we define a \underline{$2$-morphism} from $\underline{\phi}$ to $\underline{\psi}$ to be a $C^\infty$-derivation\footnote{Note that $\phi^\sharp=\psi^\sharp$ implies that $\phi'$, $\psi'$ define the same $\mathcal O'_Y$-module structure on $\phi_*(\mathcal E_X)$.}
	\begin{equation}\eta:\mathcal O'_Y\rightarrow\phi_*(\mathcal E_X),\end{equation} 
lifting $\psi'-\phi'$ and $\psi''-\phi''$, i.e. making the following diagram commutative:
	\begin{equation}\label{TwoMorphism}\xymatrix{\mathcal E_Y\ar[d]_d\ar[rr]^{\psi''-\phi''} && \phi_*(\mathcal E_X)\ar[d]^{\phi_*(d)}\\
	\mathcal O'_Y\ar[rr]_{\psi'-\phi'}\ar[rru]^\eta && \phi_*(d\mathcal E_X)}\end{equation}
Consider a diagram of $1$- and $2$-morphisms:
	\begin{equation}\xymatrix{\underline{X}\ar@{=}[d]\ar[rr]^{\underline{\phi}_1\qquad} & \ar[d]^{\eta_1} & 
	\underline{Y}\ar@{=}[d]\ar[rr]^{\underline{\phi}_2\qquad} & \ar[d]^{\eta_2} & \underline{Z}\ar@{=}[d]\\
	\underline{X}\ar@{=}[d]\ar[rr]^{\underline{\chi}_1\qquad} & \ar[d]^{\theta_1} & \underline{Y}\ar@{=}[d]\ar[rr]^{\underline{\chi}_2\qquad} & 
	\ar[d]^{\theta_2} & \underline{Z}\ar@{=}[d]\\
	\underline{X}\ar[rr]^{\underline{\psi}_1\qquad} && \underline{Y}\ar[rr]^{\underline{\psi}_2\qquad} && \underline{Z}}\end{equation}
It is immediate from (\ref{TwoMorphism}), that $\eta_1+\theta_1$ is a $2$-morphism from $\underline{\phi}_1$ to $\underline{\psi}_1$. Thus we define the \underline{vertical composition} as follows:
	\begin{equation}\theta_1\circ\eta_1:=\eta_1+\theta_1,\quad\theta_2\circ\eta_2:=\eta_2+\theta_2.\end{equation}
It is easy to see that $(\phi_2)_*(\eta_1)\circ\phi'_2$ and $(\phi_2)_*(\chi''_1)\circ\eta_2$ are $2$-morphisms from $\underline{\phi}_2\circ\underline{\phi}_1$ to $\underline{\phi}_2\circ\underline{\chi}_1$ and from $\underline{\phi}_2\circ\underline{\chi}_1$ to $\underline{\chi}_2\circ\underline{\chi}_1$ respectively. 
Therefore, we define the \underline{horizontal composition} as follows:
	\begin{equation}\eta_2\square\eta_1:=(\phi_2)_*(\eta_1)\circ\phi'_2+(\phi_2)_*(\chi''_1)\circ\eta_2,\end{equation}
	$$\theta_2\square\theta_1:=(\chi_2)_*(\theta_1)\circ\chi'_2+(\chi_2)_*(\psi''_1)\circ\theta_2.$$
It is straightforward to check, that $\square$ is associative and unital, and the interchange condition $(\theta_2\square\theta_1)\circ(\eta_2\square\eta_1)=(\theta_2\circ\eta_2)\square(\theta_1\circ\eta_1)$ is satisfied. 
Thus $d$-spaces, $1$- and $2$-morphisms form a strict $2$-category, that we will denote by $\underline{G}$. 

It is easy to check that $\underline{G}$ is equivalent to the $2$-category of $d$-spaces, defined in \cite{Jo12}. The difference is only in the presentation: in \cite{Jo12} derivations are written as maps out of sheaves of differential forms, and instead of pushforwards, one uses pullbacks of sheaves. 

We will denote by $\underline{G}_{ft}\subset\underline{G}$ the full $2$-subcategory, consisting of $d$-spaces of finite type, i.e. those $(X,\mathcal O'_X,\mathcal E_X,d)$ s.t. $\mathcal O_X:=\mathcal O'_X/d\mathcal E_X$ is a sheaf of finitely generated $C^\infty$-rings.

\subsection{Truncation of cosimplicial $C^\infty$-schemes}\label{TruncationSection}

Let $\mathbf X:=(X,\mathcal O_{\bullet,X})\in\mathbf G_{ft}$. We would like to define a $d$-space, that is the truncation of $\mathbf X$. First we need to recall some standard notation from the theory of simplicial modules: let $d^k_i:\mathcal O_{k,X}\rightarrow\mathcal O_{k-1,X}$ be the boundaries, one defines
	\begin{equation}\forall k\geq 1,\qquad\mathcal N_k:=\underset{0\leq i\leq k-1}\bigcap Ker(d^k_i)\subset \mathcal O_{k,X}.\end{equation}
Clearly, this is a sequence of sheaves of ideals. One has $d^{k-1}_{k-1}\circ d^k_k=0$ on $\mathcal N_*$, and hence $(\mathcal N_k,(-1)^kd^k_k)$ is a complex, called \underline{the normalized complex}.

Now we define two sheaves on $X$:
	\begin{equation}\mathcal O'_X:=\mathcal O_{0,X}/(d^1_1(\mathcal N_1))^2,\qquad \mathcal E:=\mathcal N_1/(d^2_2(\mathcal N_2)+\mathcal N_1^2).\end{equation}
Since $d^1_1\circ d^2_2=0$ on $\mathcal N_*$, it is clear that $d^1_1$ induces a morphism of $\mathcal O'_X$-modules 
	\begin{equation}d:\mathcal E\rightarrow\mathcal O'_X.\end{equation}
\begin{proposition} Defined as above $(X,\mathcal O'_X,\mathcal E,d)$ is a $d$-space, and the assignment $\mathbf X\mapsto (X,\mathcal O'_X,\mathcal E,d)$ extends to a functor $\mathbf T:\mathbf G_{ft}\rightarrow\underline{G}_{ft}$.\end{proposition}
\begin{proof} Since cohomology of a normalized complex is isomorphic to the sequence of homotopy groups of the original simplicial module, it is clear that $\mathcal O'_X/d\mathcal E\cong\mathcal O_X:=\pi_0(\mathcal O_{\bullet,X})$. Therefore, to prove that $(X,\mathcal O'_X,\mathcal E,d)$ is a $d$-space, it is enough to show that
	\begin{equation}(d\mathcal E)\mathcal E=0.\end{equation}
Let $U\subseteq X$ be open, and let $a_1,a_2\in\Gamma(U,\mathcal N_1)$. Then 
	\begin{equation}s_0(a_1)s_1(a_2)-s_1(a_1a_2)\in\Gamma(U,\mathcal N_2),\footnote{Here $s_0,s_1:\mathcal O_{1,X}\rightarrow\mathcal O_{2,X}$ are the two degenerations.}\end{equation} 
and clearly 
	\begin{equation}d^2_2(s_0(a_1)s_1(a_2)-s_1(a_1a_2))=d^1_1(a_1)a_2-a_1a_2,\end{equation}
i.e. the class of $d^1_1(a_1)a_2$ in $\mathcal E$ is $0$.

From functoriality of the normalized complex, it is clear that a morphism $\mathbf X\rightarrow\mathbf Y$ induces a $1$-morphism $\mathbf T(\mathbf X)\rightarrow\mathbf T(\mathbf Y)$, and this assignment is functorial.\end{proof}

\smallskip

From the simplicial enrichment of $\mathbf G_{ft}$ we can obtain a $2$-category as follows. Let $\mathbf G_{ft}^f\subset\mathbf G_{ft}$ be the full subcategory, consisting of fibrant schemes. Then, since every object in $\mathbf G_{ft}$ is cofibrant, for any $\mathbf X,\mathbf Y\in\mathbf G^f_{ft}$ the simplicial set $\underline{Hom}(\mathbf X,\mathbf Y)$ is fibrant. 

Let $\underline{\mathbf G}_{ft}$ be the $2$-category consisting of the same objects and morphisms as $\mathbf G_{ft}^f$, and with $2$-morphisms being homotopy classes of $1$-simplices in mapping spaces of $\mathbf G_{ft}^f$. Clearly $\underline{\mathbf G}_{ft}$ is a $2$-category, where each $2$-morphism is invertible.
\begin{proposition} The truncation functor $\mathbf T$, defined above, extends to a $2$-functor
	\begin{equation}\xymatrix{\underline{\mathbf G}_{ft}\ar[r] & \underline{G}_{ft}.}\end{equation}
\end{proposition}
\begin{proof} Let $\alpha^0,\alpha^1:\mathbf X\rightarrow\mathbf Y$ be two $1$-morphisms in $\underline{\mathbf G}_{ft}$, and let $\beta:\mathbf X\rightarrow\mathbf Y^{\Delta[1]}$ be a $2$-morphism from $\alpha^0$ to $\alpha^1$. This means that $i_0,i_1:\Delta[0]\rightarrow\Delta[1]$ define a commutative diagram
	\begin{equation}\label{SimplicialHomotopy}\xymatrix{& \mathbf X\ar[d]^\beta\ar[ld]_{\alpha^0}\ar[rd]^{\alpha^1} &\\ 
	\mathbf Y & \mathbf Y^{\Delta[1]}\ar[l]^{i^*_0}\ar[r]_{i_1^*} & \mathbf Y}\end{equation}
By definition $\mathbf Y^{\Delta[1]}=(Y,\mathcal O_{\bullet,Y}\otimes\Delta[1])$, and the first two levels of $\mathcal O_{\bullet,Y}\otimes\Delta[1]$ are
	\begin{equation}\mathcal O^0_{0,Y}\coprod\mathcal O^1_{0,Y},\qquad
	\mathcal O_{1,Y}^0\coprod\mathcal O_{1,Y}^1\coprod\mathcal O_{1,Y}^2,\end{equation}
where coproducts are taken in the category of sheaves of $C^\infty$-rings, and the superscript stands for indexing by simplices of $\Delta[1]$. Consider the following map
	\begin{equation}\xymatrix{\nu:\mathcal O_{0,Y}\ar[r]^s & \mathcal O^1_{1,Y}\ar[r]^{\beta^\sharp} & 
	\beta_*(\mathcal O_{1,X}),}\end{equation}
where $s:\mathcal O_{0,Y}\rightarrow\mathcal O_{1,Y}$ is the degeneration. Recall that $\mathcal N_1$ is a direct summand of $\mathcal O_{1,X}$, therefore, composing  with the projection, we get
	\begin{equation}\label{AlmostReduction}\xymatrix{\mathcal O_{0,Y}\ar[r]^{\nu\quad} & \beta_*(\mathcal O_{1,X})\ar[r] & 
	\beta_*(\mathcal N_1)\ar[r] & \beta_*(\mathcal N_1/(d_2\mathcal N_2+\mathcal N_1^2)).}\end{equation}
Since in the second projection we divide by $\mathcal N_1^2$, it is clear that (\ref{AlmostReduction}) factors through $\mathcal O_{0,Y}/(d^1_1\mathcal N_1)^2$, and hence we obtain
	\begin{equation}\xymatrix{\eta:\mathcal O_{0,Y}/(d^1_1\mathcal N_1)^2\ar[r] & 
	\beta_*(\mathcal N_1/(d^2_2\mathcal N_2+\mathcal N_1^2)).}\end{equation}
It is straightforward to check that $\eta$ is a $2$-morphism from $\mathbf{T}(\alpha^0)$ to $\mathbf{T}(\alpha^1)$, and moreover, the composition 
	\begin{equation}\xymatrix{Hom(\mathbf X,\mathbf Y^{\Delta[1]})\times 
	Hom(\mathbf Y,\mathbf Z^{\Delta[1]})\ar[r] &
	Hom(\mathbf X,\mathbf Z^{\Delta[1]}),}\end{equation}
given by the diagonal $\Delta[1]\rightarrow\Delta[1]\times\Delta[1]$, corresponds to the horizontal composition in $\underline{G}_{ft}$. It is immediate to see that $\mathbf{T}$ maps vertical composition in $\mathbf{G}_{ft}$ to the vertical composition in $\underline{G}_{ft}$.\end{proof}

\smallskip

Since we consider only fibrant cosimplicial $C^\infty$-schemes, a morphism in $\mathbf{G}_{ft}^f$ is a weak equivalence, if and only if it has a quasi-inverse. Therefore, it is clear that $\mathbf T$ maps weak equivalences to equivalences. 

\subsection{D-manifolds}\label{DManifoldsSection}

Let $\mathbf X\in\mathbf G_{ft}$ be a standard Kuranishi neighbourhood (Definition \ref{StandardKuranishi}), i.e. $\mathbf X$ is a homotopy pullback in $\mathbf G_{ft}$ of
	\begin{equation}\xymatrix{&& \mathbb R^0\ar[d]\\ \mathbb R^n\ar[rr]_\nu && \mathbb R^m}\end{equation}
We know that $\mathbf X=\mathbf{Sp}(A_\bullet)$, where $A_\bullet$ is a $1$-skeletal simplicial $C^\infty$-ring, with 
	\begin{equation}A_0=C^\infty(\mathbb R^n),\quad A_1=C^\infty(\mathbb R^{n+m}).\end{equation}
Choosing coordinates $\{x_1,\ldots,x_n\}$, $\{y_1,\ldots,y_m\}$ on $\mathbb R^n$ and $\mathbb R^m$, we have that
	\begin{equation}\forall i,j,\quad d^1_0(y_i):=0,\quad d^1_0(x_j):=x_j,\quad 
	d^1_1(y_i):=f_i,\quad d^1_1(x_j):=x_j,\end{equation}
where $f_i=\nu^*(y_i)$. It is easy to see that
	\begin{equation}A_2\supset\mathcal N_2:=Ker(d^2_0)\cap Ker(d^2_1)=Ker(d^2_0)\cdot Ker(d^2_1),\end{equation}
and hence
	\begin{equation}d^2_2(\mathcal N_2)=\underset{1\leq i,j\leq m}\Sigma(y_i(y_j-f_j)).\end{equation}
Therefore
	\begin{equation}\mathcal E:=\mathcal N_1/(d^2_2(\mathcal N_2)+\mathcal N_1^2)=\end{equation}
	$$=(\underset{1\leq i\leq m}\Sigma(y_i)/\underset{1\leq i,j\leq m}\Sigma(y_i y_j))\underset{C^\infty(\mathbb R^n)}\coprod 
	C^\infty(\mathbb R^n)/\underset{1\leq i\leq m}\Sigma(f_i).$$
In other words, $\mathcal E$ is obtained by taking the bundle of vertical forms on $\xi:\mathbb R^{m+n}\rightarrow\mathbb R^n$ at the $0$-section, and then restricting it to the subscheme $\{\nu=0\}\subseteq\mathbb R^n$. Since fibers of $\xi$ are linear spaces, the bundle of vertical vector fields at the $0$-section is naturally identified with $\xi$ itself, and hence $\mathcal E$ is naturally isomorphic to the bundle, obtained by restricting $\xi^*$ to $\{\nu=0\}$. 

\smallskip

Given an $\mathbf X$ as above, we will call its truncation $\mathbf T(\mathbf X)$ \underline{a standard model} of a $d$-manifold. This obviously agrees with \cite{Jo12}, Definition 5.13, and hence we define \underline{$d$-manifolds} to be $d$-spaces, that are locally equivalent in $\underline{G}$ to standard models. This definition is different from \cite{Jo12} only in that we do not require equidimensionality. We denote by $\underline{Man}_{ft}\subset\underline{G}_{ft}$ the full $2$-subcategory, consisting of $d$-manifolds of finite type.

\smallskip

Recall that $\mathbf{Man}_{ft}\subset\mathbf G_{ft}$ is the full simplicial subcategory, consisting of derived manifolds of finite type. Let $\mathbf{Man}^f_{ft}\subset\mathbf{Man}_{ft}$ be the full simplicial subcategory, consisting of fibrant derived manifolds of finite type. Correspondingly, let $\underline{\mathbf{Man}}_{ft}\subset\underline{\mathbf G}_{ft}$ be the full $2$-subcategory, consisting of derived manifolds of finite type. Since $\underline{\mathbf{Man}}_{ft}\subset\underline{\mathbf{G}}_{ft}$ consists of objects that are locally equivalent to Kuranishi neighbourhoods, it is clear that the $2$-functor $\mathbf T$ maps $\underline{\mathbf{Man}}_{ft}$ to $\underline{Man}_{ft}$.

\smallskip

We would like to investigate the properties of 
	\begin{equation}\xymatrix{\mathbf T:\underline{\mathbf{Man}}_{ft}\ar[r] & \underline{Man}_{ft},}\end{equation}
i.e. to understand if $\mathbf T$ is full, faithful, essentially surjective etc. We will not do this in full generality, but only when restricted to the full $2$-subcategories $\underline{\mathbf{Man}}_{st}\subset\underline{\mathbf{Man}}_{ft}$ and $\underline{Man}_{st}\subset\underline{Man}_{ft}$, consisting of derived manifolds and $d$-manifolds that are $0$-loci of sections of vector bundles over manifolds of finite type. For example all compact derived and $d$-manifolds are of this kind (\cite{Sp10} and \cite{Jo12} Theorems 6.28 and 6.33).

\smallskip

First of all, it is clear that every $d$-manifold in $\underline{Man}_{st}$ is equivalent to the truncation of a derived manifold in $\underline{\mathbf{Man}}_{st}$. In fact we can say more. Let $\mathfrak K$ be the category defined as follows:\begin{itemize}
\item objects are triples $(\mathcal M,E,\sigma)$, where $\mathcal M$ is a manifold of finite type, $E$ is a bundle over $\mathcal M$, and $\sigma$ is a smooth section of $E$;
\item morphisms are pairs of smooth maps $\alpha:\mathcal M\rightarrow\mathcal M'$, $\beta:E\rightarrow E$, s.t. the following diagram is commutative
	\begin{equation}\xymatrix{E\ar[d]\ar[rr]^\beta && E'\ar[d]\\ 
	\mathcal M\ar@<-.7ex>[u]_\sigma\ar@<+.7ex>[u]^0\ar[rr]_\alpha &&
	\mathcal M'\ar@<+.7ex>[u]^0\ar@<-.7ex>[u]_{\sigma'}}\end{equation}
\end{itemize}
There are two functors $\mathfrak K\rightarrow\underline{\mathbf{Man}}_{st}$ and $\mathfrak K\rightarrow\underline{Man}_{st}$. The former is given by taking the homotopy equalizer of $\sigma$ with the $0$-section, and the latter is given by restricting $E^*$ to the $0$-locus of $\sigma$ (\cite{Jo12}, Definition 5.13). As we have seen at the beginning of this section, the truncation functor completes the following commutative diagram
	\begin{equation}\xymatrix{& \mathfrak K\ar[ld]\ar[rd] &\\ 
	\underline{\mathbf{Man}}_{st}\ar[rr]_{\mathbf T} && \underline{Man}_{st}}\end{equation}
It follows immediately that the inclusion of the usual theory of manifolds into $\underline{Man}_{ft}$ factors through $\mathbf{Man}_{ft}$. Another immediate consequence is the following result.
\begin{proposition}\label{EssentiallySurjective} The truncation $2$-functor $\mathbf T:\underline{\mathbf{Man}}_{st}\rightarrow\underline{Man}_{st}$ induces a surjection between the sets of equivalence classes of objects.\end{proposition}
Next we investigate the question of fullness. First we consider the special case, when the bundle is trivial and the manifold is just $\mathbb R^n$.
\begin{lemma}\label{StandardSurjective} Let $\mathbf X,\mathbf Y\in\mathbf{G}_{ft}$, with $\mathbf Y$ being a standard Kuranishi neighbourhood (Definition \ref{StandardKuranishi}). Then the map
	\begin{equation}\xymatrix{\mathbf T:Hom_{\mathbf G_{ft}}(\mathbf{X},\mathbf{Y})\ar[r] & 
	Hom_{\underline{G}_{ft}}(\mathbf T(\mathbf X),\mathbf T(\mathbf Y))}\end{equation}
is surjective.\end{lemma}
\begin{proof} By assumption $\mathbf{Y}=\mathbf{Sp}(\mathcal B(\mu^*,\nu^*)_\bullet)$, where 
	\begin{equation}\xymatrix{\mu^*,\nu^*:C^\infty(\mathbb R^m)\ar[r] & C^\infty(\mathbb R^n),}\end{equation} 
with $\mu^*$ being evaluation at the origin. Since $\mathbf{Sp}$ is right adjoint to $\Gamma$, we have a natural bijection
	\begin{equation}Hom_{\mathbf G_{ft}}(\mathbf X,\mathbf Y)\cong 
	Hom_{SC^\infty\mathcal R}(\mathcal B(\mu^*,\nu^*)_\bullet,\Gamma(X,\mathcal O_{\bullet,X})).\end{equation}
Let $\{x_1,\ldots,x_n\}$, $\{y_1,\ldots,y_m\}$ be the coordinates on $\mathbb R^n$, $\mathbb R^m$ respectively. Then $\mathcal B(\mu^*,\nu^*)_\bullet$ is a $1$-skeletal, almost free simplicial $C^\infty$-ring, with $x$, $y$ being the almost free generators. This means that the set of morphisms $\mathcal B(\mu^*,\nu^*)_\bullet\rightarrow\Gamma(X,\mathcal O_{\bullet,X})$ is in bijective correspondence with the set of assignments
	\begin{equation}\xymatrix{x_i\ar@{|->}[r] & f_i\in\Gamma(X,\mathcal O_{0,X}), & 
	y_j\ar@{|->}[r] & g_j\in\Gamma(X,\mathcal O_{1,X}),}\end{equation}
such that
	\begin{equation}d^1_0(g_j)=0,\quad d^1_1(g_j)=\nu^*_j(f_1,\ldots,f_n),\end{equation}
where $\nu^*_j:=\nu^*(y_j)\in C^\infty(\mathbb R^n)$.

Let $\underline{\phi}\in Hom_{\underline{G}_{ft}}(\mathbf T(\mathbf X),\mathbf T(\mathbf Y))$. Since $C^\infty(\mathbb R^n)$ is a free $C^\infty$-ring, we can find a $\phi_0$, making the following diagram commutative:
	\begin{equation}\xymatrix{C^\infty(\mathbb R^n)\ar@{->>}[r]\ar@{.>}[d]^{\phi_0} & 
	\Gamma(Y,\mathcal O_{0,Y}/(d^1_1\mathcal N_1)^2)\ar[d]^{\underline{\phi}}\\
	\Gamma(X,\mathcal O_{0,X})\ar@{->>}[r] & \Gamma(X,\mathcal O_{0,X}/(d^1_1\mathcal N_1)^2)}\end{equation}
Similarly, we can find $\widetilde{\phi}_1$, making the following diagram commutative:
	\begin{equation}\xymatrix{C^\infty(\mathbb R^m)\ar[r]\ar@{.>}[d]_{\widetilde{\phi}_1} & 
	\Gamma(Y,\mathcal N_1/(d^2_2\mathcal N_2+\mathcal N_1^2))\ar[d]^{\underline{\phi}}\\
	\Gamma(X,\mathcal N_1)\ar@{->>}[r] & \Gamma(X,\mathcal N_1/(d^2_2\mathcal N_2+\mathcal N_1^2))}\end{equation}
It might happen that $d^1_1(\widetilde{\phi}_1(y_j))\neq\nu^*_j(\phi_0(x_1),\ldots,\phi_0(x_n))$, however, clearly
	\begin{equation}\nu^*_j(\phi_0(x_1),\ldots,\phi_0(x_n))-d^1_1(\widetilde{\phi}_1(y_j))\in (d^1_1\mathcal N_1)^2.\end{equation}
Therefore, we can choose $\epsilon_j\in\Gamma(X,\mathcal N_1^2)$, s.t.
	\begin{equation}d^1_1(\widetilde{\phi}_1(y_j)+\epsilon_j)=\nu^*_j(\phi_0(x_1),\ldots,\phi_0(x_n)).\end{equation}
Thus, defining
	\begin{equation}\xymatrix{x_i\ar@{|->}[r] & \phi_0(x_i), & y_j\ar@{|->}[r] & \widetilde{\phi}_1(y_j)+\epsilon_j}\end{equation}
we are done.\end{proof}

\smallskip

To extend Lemma \ref{StandardSurjective} to the cases where the codomain is not standard Kuranishi, we will need to glue morphisms, using softness of the structure sheaves. The following result provides the means for this.
\begin{lemma}\label{SoftCorrection} Let $(\phi,\phi^\sharp):\mathbf X\rightarrow\mathbf Y$ be a morphism in $\mathbf G_{ft}$, and suppose that $\mathbf Y$ is a standard Kuranishi neighbourhood. For each open $U\subseteq Y$ let
	\begin{equation}\mathcal F(U):=\{\psi^\sharp:\mathcal O_{\bullet,Y}|_U\rightarrow\phi_*(\mathcal O_{\bullet,X})|_U\text{ s.t. }
	\mathbf T(\phi|_U,\psi^\sharp)=\mathbf T(\phi,\phi^\sharp)|_U\}.\end{equation}
Then $U\mapsto\mathcal F(U)$ is a soft sheaf (of sets) on $Y$.\end{lemma}
\begin{proof} By assumption $\mathbf{Y}=\mathbf{Sp}(\mathcal B(\mu^*,\nu^*)_\bullet)$, where 
	\begin{equation}\xymatrix{\mu^*,\nu^*:C^\infty(\mathbb R^m)\ar[r] & C^\infty(\mathbb R^n),}\end{equation} 
with $\mu^*$ being evaluation at the origin. Let $x_1,\ldots,x_n$, $y_1,\ldots,y_m$ be the coordinates on $\mathbb R^n$, $\mathbb R^m$, and let $\nu^*_j:=\nu^*(y_j)\in C^\infty(\mathbb R^n)$. 

Let $V\subseteq Y$ be closed, we denote by $\{[x_i]_V\}$, $\{[y_j]_V\}$ the images of $\{x_i\}$, $\{y_j\}$ in $\Gamma(V,\mathcal O_{0,Y})$, $\Gamma(V,\mathcal O_{1,Y})$ respectively.

Let $\psi^\sharp\in\mathcal F(V)$, clearly
	\begin{equation}\psi^\sharp([x_i]_V)-\phi^\sharp([x_i]_V)\in\Gamma(V,\phi_*((d^1_1\mathcal N_1)^2)),\end{equation}
and since $\phi_*((d^1_1\mathcal N_1)^2)$ is a soft sheaf, there are
	\begin{equation}\alpha_i\in\Gamma(Y,\phi_*((d^1_1\mathcal N_1)^2))\end{equation}
extending $\psi^\sharp([x_i]_V)-\phi^\sharp([x_i]_V)$. Let $\gamma_i:=\phi^\sharp([x_i]_Y)+\alpha_i$ and consider the sheaves
	\begin{equation}(d^1_1)^{-1}(\nu^*_j(\gamma_1,\ldots,\gamma_n)-d^1_1(\phi^\sharp([y_j]_Y)))\subseteq
	\phi_*(d^2_2\mathcal N_2+\mathcal N_1^2).\end{equation}
These are soft sheaves as well, and hence there are
	\begin{equation}\beta_j\in\Gamma(Y,\phi_*(d^2_2\mathcal N_2+\mathcal N_1^2)),\end{equation}
extending 
	\begin{equation}\psi^\sharp([y_j])-\phi^\sharp([y_j])\in\Gamma(V,\phi_*(d^2_2\mathcal N_2+\mathcal N_1^2)),\end{equation}
and such that
	\begin{equation}d^1_1(\phi^\sharp([y_j]_Y)+\beta_j)=\nu^*_j(\gamma_1,\ldots,\gamma_n).\end{equation}
Now we define an element of $\mathcal F(Y)$ as follows:
	\begin{equation}\xymatrix{[x_i]_Y\ar@{|->}[r] & \phi^\sharp([x_i]_Y)+\alpha_i, & 
	[y_j]_Y\ar@{|->}[r] & \phi^\sharp([y_j]_Y)+\beta_j.}\end{equation}
Clearly this is an extension of $\psi^\sharp$.\end{proof}

\smallskip

To use lemma \ref{SoftCorrection} we need to modify Lemma \ref{StandardSurjective}, so that the codomains are not necessarily open Kuranishi neighbourhoods, but closed ones. We do this in the following lemma. The proof is straightforward.
\begin{lemma}\label{ClosedKuranishi} Let $\mathbf X,\mathbf Y\in\mathbf G_{ft}$, let $\iota:Z\subseteq Y$ be a closed subset, and let 
	\begin{equation}\mathbf Y|_Z:=(Z,\iota^{-1}(\mathcal O_{\bullet,Y})).\end{equation}
Suppose that the map
	\begin{equation}\xymatrix{\mathbf T:Hom_{\mathbf G_{ft}}(\mathbf X,\mathbf Y)\ar[r] & 
	Hom_{\underline{G}_{ft}}(\mathbf T(\mathbf X),\mathbf T(\mathbf Y))}\end{equation}
is surjective. Then 
	\begin{equation}\xymatrix{\mathbf T:Hom_{\mathbf G_{ft}}(\mathbf X,\mathbf Y|_Z)\ar[r] & 
	Hom_{\underline{G}_{ft}}(\mathbf T(\mathbf X),\mathbf T(\mathbf Y|_Z))}\end{equation}
is surjective as well.\end{lemma}
Now we are ready to show that $\mathbf T$ is almost $1$-full. The following proposition makes this statement precise.
\begin{proposition}\label{OneFullness} Let $\mathbf X\in\underline{\mathbf{G}}_{ft}$, $\mathbf Y\in\underline{\mathbf{Man}}_{st}$, and let 
	\begin{equation}\underline{\phi}:\mathbf T(\mathbf X)\rightarrow\mathbf T(\mathbf Y)\end{equation} 
be a $1$-morphism between the corresponding truncations. There is a morphism $\Phi:\mathbf{X}\rightarrow\mathbf{Y}$, s.t.
	\begin{equation}\mathbf T(\Phi)\cong\underline{\phi}.\end{equation}
\end{proposition}
\begin{proof} By assumption $\mathbf Y$ is the homotopy equalizer in the diagram
	\begin{equation}\xymatrix{\mathcal M\ar@<+.5ex>[r]^0\ar@<-.5ex>[r]_\sigma & E}\end{equation}
where $\mathcal M$ is a manifold of finite type, $E$ is a vector bundle over $\mathcal M$, and $0,\sigma$ are sections of $E$. Therefore, according to Proposition \ref{ZeroLocus}, we can assume that 
	\begin{equation}\mathbf Y=\mathbf{Sp}(\{C^\infty(E^{\times_\mathcal M^k})\}_{k\geq 0}).\end{equation}
Since $\mathcal M$ is second countable, we can cover $\mathcal M$ with a countable family of closed subsets $\{\mathcal M_i\}$, s.t. over each $\mathcal M_i$ the bundle $E$ is trivial. Then we have $\mathbf Y=\underset{i\in\mathbb N}\bigcup \mathbf Y_i$, where $\mathbf Y_i$ is the $0$-locus of $\sigma:\mathcal M_i\rightarrow E_i$, and hence each $\mathbf Y_i$ is a closed subscheme of a standard Kuranishi neighbourhood.	Let $\{\mathbf X_i\}$ be the pre-images of $\{\mathbf Y_i\}$.

Now we define a partially ordered set $P$ as follows. Elements of $P$ are morphisms 
	\begin{equation}\xymatrix{\Phi_k:\underset{1\leq i\leq k}\bigcup\mathbf X_i\ar[r] & 
	\underset{1\leq i\leq k}\bigcup\mathbf Y_i, & k\in\mathbb N\cup\{\infty\},}\end{equation}
lifting the corresponding restriction of $\underline{\phi}$, and $\Phi_k\geq\Phi_l$ if $k\geq l$, and the restriction of $\Phi_k$ to $\underset{1\leq i\leq l}\bigcup\mathbf X_i$ equals $\Phi_l$. It is clear that the set $P$ satisfies the conditions of Zorn lemma, and hence it has maximal elements. If the maximal element has index $\infty$ (and in particular an $\infty$-indexed element exists), we are done, since we have found a lift of $\underline{\phi}$. 

Suppose a maximal element is $\Phi_k$, with $k<\infty$. From Lemma \ref{StandardSurjective} and Lemma \ref{ClosedKuranishi} we know that
	\begin{equation}\xymatrix{Hom_{\mathbf{G}_{ft}}(\mathbf X_{k+1},\mathbf Y_{k+1})\ar[r] &
	Hom_{\underline{G}_{ft}}(\mathbf T(\mathbf X_{k+1}),\mathbf T(\mathbf Y_{k+1}))}\end{equation}
is surjective, hence $\underline{\phi}|_{\mathbf X_{k+1}}$ has a lift. From Lemma \ref{SoftCorrection} we know that we can choose a lift that agrees with $\Phi_k$, restricted to $(\underset{1\leq i\leq k}\bigcup\mathbf X_i)\cap\mathbf X_{k+1}$. Thus $\Phi_k$ cannot be a maximal element.\end{proof}

\smallskip

It remains to show that $\mathbf T$ detects equivalences. This is done in the following proposition.
\begin{proposition}\label{Detection} Let $\Phi:\mathbf X\rightarrow\mathbf Y$ be a morphism in $\underline{\mathbf{Man}}_{ft}$, s.t. $\mathbf T(\Phi)$ is an equivalence. Then $\Phi$ is a weak equivalence.\end{proposition}
\begin{proof} For $\Phi$ to be a weak equivalence is a local property, since $\mathbf T(\Phi)$ being an equivalence already ensures that the underlying map of topological spaces is a homeomorphism. Therefore, we can assume that $\mathbf X,\mathbf Y$ are germs of derived manifolds. 

So let $\nu:\mathbb R^n\rightarrow\mathbb R^{m+n}$, $\nu':\mathbb R^{n'}\rightarrow\mathbb R^{m'+n'}$ be some sections of trivial bundles, and suppose $\mathbf X$, $\mathbf Y$ are germs of $\{\nu=0\}$, $\{\nu'=0\}$ respectively. We can always find weakly equivalent minimal presentations, i.e. those with minimal $n$ and $n'$. For minimal $n$, $n'$ one can show, as in \cite{Jo12} section 5.3, that  $\mathbf T(\Phi):\mathbf T(\mathbf X)\rightarrow\mathbf T(\mathbf Y)$ being an equivalence implies that $\Phi$ is an isomorphism.\end{proof}

\smallskip

Put altogether, Propositions \ref{EssentiallySurjective}, \ref{OneFullness}, \ref{Detection} give the following statement.
\begin{theorem} The truncation $2$-functor
	\begin{equation}\xymatrix{\mathbf T:\underline{\mathbf{Man}}_{st}\ar[r] & \underline{Man}_{st}}\end{equation} 
induces a full and essentially surjective $1$-functor between the $1$-categories, obtained by identifying isomorphic $1$-morphisms. Moreover, the map between the equivalence classes of objects is a bijection.\end{theorem}
This $1$-functor is, however, not faithful, as shown in the following example. Let $\nu:\mathbb R^2\rightarrow\mathbb R^3$ be defined by
	\begin{equation}\nu^*(u_1):=x^2,\quad\nu^*(u_2):=y^2,\quad\nu^*(u_3):=x y,\end{equation}
where $x,y$ and $u_1,u_2,u_3$ are the coordinates. Let $E:=\mathbb R^5\rightarrow\mathbb R^2$ be the trivial bundle, and let $\mathbf Y$ be the homotopy equalizer of
	\begin{equation}\xymatrix{\mathbb R^2\ar@<+.5ex>[r]^0\ar@<-.5ex>[r]_\nu & E.}\end{equation}
According to Proposition \ref{ZeroLocus} we can write $\mathbf Y$ as $\mathbf{Sp}(A_\bullet)$, where 
	\begin{equation}A_\bullet:=\{C^\infty(E^{\times_{\mathbb R^3}^k})\}_{k\geq 0}.\end{equation}
It is easy to see that $u_1u_2-u^2_3\in Ker(d^1_0)\cap Ker(d^1_1)\subset A_1$. Moreover, the class of $u_1u_2-u_3^2$ in $\pi_1(A_\bullet)$ is not trivial. Yet, $u_1u_2-u_3^2$ vanishes in the truncation to $d$-manifolds, since obviously
	\begin{equation}u_1u_2-u_3^2\in(Ker(d^1_0))^2\subset A_1.\end{equation}
Let $\mathbf X$ be the homotopy equalizer of
	\begin{equation}\xymatrix{\mathbb R^0\ar@<+.5ex>[r]^0\ar@<-.5ex>[r]_0 & \mathbb R,}\end{equation}
and let $t$ be the coordinate on $\mathbb R$. We define two morphisms $\Phi,\Psi:\mathbf Y\rightarrow\mathbf X$ as follows:	\begin{equation}\xymatrix{\Phi:t\ar@{|->}[r] & 0, & \Psi:t\ar@{|->}[r] & u_1u_2-u_3^2.}\end{equation}
Since the class of $u_1u_2-u_3^2$ in $\pi_1(A_\bullet)$ is not trivial, clearly $\Phi$ and $\Psi$ are not homotopic in $\mathbf{Man}_{ft}$, yet $\mathbf T(\Phi)=\mathbf T(\Psi)$.

\end{document}